\documentclass[a4paper]{article}

\usepackage[utf8]{inputenc}
\usepackage[english]{babel}
\usepackage{amsmath, amsthm, amsfonts, amssymb, enumerate}
\usepackage{mathrsfs}
\usepackage[numbers]{natbib}
\usepackage{color}
\usepackage{graphicx}
\usepackage{hyperref}
\usepackage{soul}

\theoremstyle{plain}

\newtheorem{theorem}{Theorem}
\newtheorem{lemma}[theorem]{Lemma}

\newtheorem*{lemma*}{Lemma}
\newtheorem{corollary}[theorem]{Corollary}

\newtheorem*{theorem*}{Theorem}
\newtheorem{proposition}[theorem]{Proposition}

\theoremstyle{definition}
\newtheorem{definition}[theorem]{Definition}

\newtheorem{example}[theorem]{Example}

\newtheorem{remark}[theorem]{Remark}

\newtheorem*{remark*}{Remark}
\newtheorem*{notation*}{Notation}

\newcommand{\cF}{\mathscr F}

\newcommand{\N}{{\mathbb N}}
\newcommand{\R}{{\mathbb R}}
\newcommand{\Q}{{\mathbb Q}}

\newcommand{\msr}{Y}

\newcommand{\exclude}[1]{}
\definecolor{grey}{rgb}{0.7, 0.7, 0.7}
\definecolor{darkgreen}{rgb}{0., 0.6, 0.}
\definecolor{bRed}{rgb}{1,0.7,0.7}

\makeatletter
\def\blfootnote{\xdef\@thefnmark{}\@footnotetext}
\makeatother

\title{Exception sets of intrinsic and piecewise Lipschitz functions}

\author{Gunther Leobacher\thanks{G.~Leobacher is supported by the Austrian
Science Fund (FWF): Project F5508-N26, which is part of the Special Research
Program `Quasi-Monte Carlo Methods: Theory and Applications'.}\; and Alexander Steinicke}

\begin{document}

\maketitle

\begin{abstract}
We consider a class of functions defined on metric spaces which generalizes the concept of piecewise Lipschitz continuous functions on an interval or on polyhedral structures. The study of such functions requires the investigation of their exception sets where the Lipschitz property fails. The newly introduced notion of permeability describes sets which are natural
exceptions for Lipschitz continuity in a well-defined sense. One of the main results states that continuous functions which 
are intrinsically Lipschitz continuous outside a permeable set are Lipschitz continuous on the whole domain with respect to the intrinsic metric.   
We provide examples of permeable sets in $\R^d$, which include Lipschitz  submanifolds.     
\end{abstract}



\section{Introduction}

Exception sets for the regularity of a function are encountered when considering functions $f\colon I\to\R$ defined on an interval $I$ that have a certain property (e.g. continuity, differentiability) when restricted to a subset $E\subseteq I$ but not on the whole of $I$. The complement $\Theta=I\setminus E$ is then an {\em exception set} for the function's property. A particularly common example here is the case of a finite exception set $\Theta$ for a function defined on a real interval. Such a finite set partitions the interval $I$ into several parts or {\em pieces}, hence leading to the notions of {\em piecewise} continuous, differentiable, etc. functions. The property which will be of our interest is the one of Lipschitz continuity and our work is motivated by the desire to generalize the notion of `piecewise Lipschitz continuity' to the multidimensional case. Consider the following definition:

\begin{definition}\label{def:pwlip-1d}
A function $f\colon I\to \R$ is {\em piecewise Lipschitz continuous} if there exist finitely many 
points $x_1,\ldots,x_n\in I$ with  $x_0:=\inf I<x_1<\ldots <x_{n+1}:=\sup I$,  
such that $f|_{(x_{j-1},x_j)}$ is Lipschitz continuous for every $j=1,\dots,n+1$.
\end{definition}

With this definition, the following result is easily proven (see, for example, Lemma 2.4 in \cite{sz2016b}):

\begin{lemma}\label{lem:pwlip+cont=lip}
Let $I\subseteq \R$ be an interval and  let $f\colon I\to \R$ be continuous and piecewise Lipschitz continuous. 
Then $f$ is Lipschitz continuous.
\end{lemma}

Definition \ref{def:pwlip-1d} is a reasonable implementation of the concept, although we do not claim that it is universally accepted across the
mathematical community. Generalizing this (or a similar) definition to the multidimensional case is far from being
unambiguous as many multidimensional concepts coincide in dimension 1: Intervals are precisely the convex 
subsets of $\R$, but also precisely the star-shaped, connected, path-connected, arc-connected sets and the polytopes.
Multidimensional generalizations of the exception set $\{x_1,\ldots,x_n\}$ in  Definition \ref{def:pwlip-1d}
are affine hyperplanes, (finite unions of) submanifolds, finite sets, among others. Classical generalizations of Definition \ref{def:pwlip-1d} to higher dimensions require the desired property on elements of a polytopal, polyhedral or simplicial subdivision of the domain. Variants of this procedure are well known for defining the class of {\em piecewise linear} (\textsc{pl}) or {\em piecewise differentiable} (\textsc{pdiff}) functions, see e.g. \cite[1.4, Ch.1]{Rourke1972}, \cite[Section 2.2]{scholtes2012introduction} or \cite[Section 3.9]{thurston97}. However, these classes comprise of  continuous functions. A definition of piecewise linear using a subdivison by hyperplanes not implying continuity can be found in the introduction of \cite{Conn1998}. 
A simple extension of the notion `piecewise Lipschitz continuous', loosely following the ideas in the references above,
could be given by the subsequent definition:
\begin{definition}\label{def:pwlip-md}
Let $M\subseteq \R^d$. 
A function $f\colon M\to \R$ is {\em piecewise Lipschitz continuous} on $M$ if there exist finitely many 
open polyhedra\footnote{By a polyhedron we mean the intersection of finitely many affine half spaces in $\R^d$.} $P_1,\ldots,P_m$ with $P_j\cap P_k=\emptyset$ for $j\ne k$ and $M\subseteq\bigcup_{j=1}^n \overline{P}_j$,
such that $f|_{P_j}$ is Lipschitz continuous for every $j=1,\dots,n$.
\end{definition}

The definition in \cite{mikkelsen18} (implicit in Assumption 2.2)
is similar to Definition \ref{def:pwlip-md}, but replaces `open polyhedra' by `open sets'. In \cite{balcan18}
general sets for the $P_i$ are allowed, and the resulting difficulties are overcome by considering
a `dispersed' family of functions which means that not too many of them have their Lipschitz-exceptions 
around the same spot. 
We will take a path  different to Definition \ref{def:pwlip-md} by concentrating on a notion 
which, instead of focussing on the {\em pieces} on which
the Lipschitz property holds, we emphasize the {\em exception set} where the Lipschitz property fails.
Related approaches can be found in 
 \cite{liwang2013}, where the authors consider functions which are 
Lipschitz on each of two parts of a domain which is split by a $C^{1,1}$ manifold. 
A more general exception set occurs
in  \cite[Section 14.2]{hlavacek2004uncertain}, where  a piecewise Lipschitz continuous function is one 
which is is defined on a union of domains  with Lipschitz boundaries, which is Lipschitz continuous on these subdomains. 
An even more complex exception set is allowed in \cite{masubuchi}, where a function is `piecewise $C^2$', if the $C^2$ property fails on a closed set of 
Lebesgue-measure 0. An interesting theorem outside the $\R^d$-setting can be found in \cite{cluckers12}, where it is shown that a specific notion 
of piecewise Lipschitz continuity follows from a local Lipschitz condition for semi-algebraic 
functions $\Q_p^d\to\Q_p$, where $\Q_p$ are the $p$-adic numbers.

\smallskip
A guiding principle for what we attempt here will be that a suitable generalization  of Lemma \ref{lem:pwlip+cont=lip} should hold.
We generalize and extend the approach of \cite{sz2016b}, who call a function $f\colon\R^d\to \R^m$  {\em piecewise Lipschitz} 
with exception set $\Theta\subseteq \R^d$, if $f|_{\R^d\setminus \Theta}$ is Lipschitz with respect to the 
intrinsic metric (see Definition \ref{def:intrinsic}) on  $\R^d\setminus \Theta$, and where $\Theta$ is a hypersurface, that is, a $(d-1)$-dimensional submanifold of $\R^d$. They prove a multidimensional version of Lemma \ref{lem:pwlip+cont=lip} under an additional condition on $\Theta$, namely, the condition that we will call
{\em finitely permeable} (see Definition \ref{def:permeable}). The task to determine suitable exception sets for Lipschitz functions with respect to the intrinsic metric should not be mixed with the -- of course related -- problem of finding sets $R$ such that functions defined on the complement $R^c$ and belonging to a certain regularity class there may be extended to the whole space. Such {\em removable} sets $R$ have been investigated in complex analysis and geometric function theory for a long time, see e.g. \cite{younsi2015}, \cite{rajala2019}, and, in connection with Lipschitz continuity, \cite{craig-et-al-2015}. In the early account \cite{Ahlfors1950}, removable sets are called {\em function theoretic null-sets} and are characterized by an extremal distance condition.

The generalization of  piecewise Lipschitz continuous functions to Lipschitz continuous functions with respect to the intrinsic metric up to an exception set includes far more functions than when the induced metric on the complement of the exception is used,
see  Example \ref{ex:intLip} in Section \ref{sec:pwlm} for instance.
Also, Lemma \ref{lem:pwlip+cont=lip} does not simply follow from  
well-known
 extension theorems, such as the classical ones of Kirszbraun (see \cite{Kirszbraun1934}, \cite[p.21]{Schwartz1969}), McShane-Whitney (see \cite{Whitney1934,McShane1934}) or more recent ones such as the one in \cite{mocanu2008}.

We expand the generalization from \cite{sz2016b} in many directions. In Section \ref{sec:pwlm} we first recall the 
notion of the intrinsic metric and introduce the concept of permeable and finitely permeable sets. This is done in the general 
framework of metric spaces. We define intrinsically Lipschitz continuous functions as functions which are Lipschitz 
continuous with respect to the intrinsic metric on their domain. Our first main result is then Theorem \ref{thm:main-th1}, 
which is a multidimensional version of Lemma \ref{lem:pwlip+cont=lip}, where the additional assumption is that the exception set is permeable. No further assumptions are required, in particular, the exception set need not be a manifold.
The proof uses transfinite induction and the Cantor-Bendixson theorem.

The notion of permeability is weaker than that of finite permeability,  and  much weaker than finiteness, and we show that 
in the 1-dimensional real case it not only is a sufficient condition on a set $\Theta$ so that every continuous function which is Lipschitz continuous with exception set  $\Theta$ is Lipschitz -- it is also necessary.

Section \ref{sec:smf} is then dedicated to finding large and practically relevant classes of subsets of $\R^d$ which 
are permeable.  We show in Theorem \ref{th:lip} that every Lipschitz submanifold which is a closed
subset of $\R^d$ is finitely permeable and thus permeable.  We discuss further generalizations, instructive examples and counterexamples.

Our research presents a new concept in analysis with already a number of non-trivial results
and generalizations. 
Moreover, it opens pathways to generalizing results in many applied fields, where concepts of piecewise Lipschitz 
continuous functions have already been used, such as image processing \cite{burger2019}, uncertain input data problems
\cite{hlavacek2004uncertain}, optimal control \cite{huo14}, stochastic differential equations \cite{sz2016b}, information processing \cite{Kuh1977,Rodriguez2019}, machine learning \cite{balcan18,sharma20a},  
dynamical systems \cite{storace2004},  
shape-from-shading problems \cite{tourin1992}.

\section{Intrinsic Lipschitz functions and permeable subsets of metric spaces}\label{sec:pwlm} 

Throughout this section, let $(M,d)$ be a metric space. To begin,  we recall some definitions for metric spaces. 
\begin{definition}[Path, arc, length]
\begin{enumerate}
\item A {\em path} in $M$ is a continuous mapping $\gamma\colon[a,b]\to M$.
We also say that $\gamma$ is a path in $M$ from $\gamma(a)$ to $\gamma(b)$. 
\item An injective path is called an {\em arc}.
\item If $\gamma\colon[a,b]\to M$ is a path in $M$, then its {\em length} $\ell(\gamma)$ is defined
as
\[
\ell(\gamma)
:=\sup\Big\{\sum_{k=1}^n d\big(\gamma(t_k),\gamma(t_{k-1})\big)\colon n\in\N,\,a=t_0<\ldots<t_n=b\Big\}\,.
\]
\end{enumerate}
\end{definition}

The following important lemma is  an immediate consequence of Proposition 3.4 and Proposition 3.5 in \cite{AlbOtt17}. It states that one can always replace a path in $M$  by an injective one with length
at most that of the original path and its image contained in that of the original path.

\begin{lemma}\label{lem:injective-curve} 
Let $x,y\in M$, $x\ne y$ and $\gamma\colon [0,1]\to M$ be a path from $x$ to $y$. 
Then there exists an arc $\eta\colon [0,1]\to M$ from $x$ to $y$ with $\eta([0,1])\subseteq \gamma([0,1])$
and $\ell(\eta)\le\ell(\gamma)$.
\end{lemma}

\begin{proof}
Consider first the case where $\gamma$ has finite length. 
Then $\gamma([0,1])$ is a continuum 
and by \cite[Proposition 3.5]{AlbOtt17} 
\[
\ell(\gamma)=\int_{\gamma([0,1])}m(\gamma,x)d\mathscr{H}^1(x)\ge\int_{\gamma([0,1])}d\mathscr{H}^1(x) =
\mathscr{H}^1\big(\gamma([0,1])\big)\,,
\]
where $m(\gamma,.)$ is the multiplicity of $\gamma$, $m(\gamma,x):=\#\big(\gamma^{-1}(\{x\})\big)$, and $\mathscr{H}^1$ is the 1-Hausdorff-measure.
Now by \cite[Proposition 3.4]{AlbOtt17}, $\gamma(0)$ and $\gamma(1)$ are connected by an arc $\eta$ in $\gamma([0,1])$ with 
$\ell(\eta)\le \mathscr{H}^1\big(\gamma([0,1])\big)\le \ell(\gamma)$.

If $\ell(\gamma)=\infty$, the assertion follows immediately from the Hahn-Mazurkiewicz theorem, see \cite[Section 31]{willard2004general}.
\end{proof}

\begin{definition}[Intrinsic metric, length space, quasi-convexity]\label{def:intrinsic}
Let $E\subseteq M$ and $\Gamma(x,y)$ be the set of all paths of finite length in $E$ from $x$ to $y$. 
The {\em intrinsic metric} $\rho_E$ on $E$ is defined by
\[
\rho_E(x,y):=\inf\big\{\ell(\gamma)\colon \gamma\in \Gamma(x,y)\big\}\,, \qquad(x,y\in E)\,,
\]
with the convention $\inf\emptyset=\infty$.
The metric space $(M,d)$ is a {\em length space}, iff $\rho_M=d$.
We call  $(M,d)$  {\em $C$-quasi-convex} iff there exists $C>0$ s.t.~$\rho_M(x,y)\le C d(x,y)$ for all $x,y\in M$. 
 \end{definition}

Note that $\rho_E$ is not a proper metric in that it may take the value infinity. 
Of course, one could relate $\rho_E$ to a proper metric 
\[
\tilde \rho_E (x,y)
:=\begin{cases}
\frac{\rho_E(x,y)}{1+\rho_E(x,y)}\,, &\text{if } \rho_E(x,y)<\infty\,,\\
1\,, &\text{if }  \rho_E(x,y)=\infty\,.
\end{cases}
\] 
However, we stick to $\rho_E$, as it is the more natural choice and the extended co-domain does not lead to any difficulties.

It is readily checked that, if we allow $\infty$ as the value of a metric, then  
$(E,\rho_E)$ is  a length space.

\smallskip

See \cite[Section 7]{heinonen01} and \cite{hakobyan2008} for interesting consequences of quasi-convexity in the context of Lipschitz analysis.

\begin{definition}[Intrinsically Lipschitz continuous function]\label{def:pw-lip}
Let $E\subseteq M$, let $(\msr,d_\msr)$ be a metric space
and $f\colon M\to \msr$ a function.
\begin{enumerate}
\item We call  $f$  {\em intrinsically $L$-Lipschitz continuous} 
on  $E$ 
iff
$f|_E\colon E\to \msr$ is Lipschitz continuous with respect to the
intrinsic metric $\rho_E$ on $E$ and $d_\msr$ on $\msr$ and Lipschitz constant $L$.
\item We call  $f$  {\em intrinsically Lipschitz continuous} 
on  $E$ 
iff
$f$ is intrinsically $L$-Lipschitz continuous for some $L$.
\item In the above cases we call $M\setminus E$ an {\em exception set (for intrinsic Lipschitz continuity)} of $f$. 
\end{enumerate}
\end{definition}


\begin{example}\label{ex:intLip}
Consider the function
$f\colon\R^2\longrightarrow\R$, $f(x)=\|x\|\arg(x)$. Then $f$ is not Lipschitz continuous with respect to the induced metric, since 
$\lim_{h\to 0+}f(\cos(\pi-h),\sin(\pi-h))=-\pi$ and 
$\lim_{h\to 0+}f(\cos(\pi+h),\sin(\pi+h))=\pi$ for $x_1<0$.

It is readily checked, however, that $f$ is Lipschitz continuous on $E=\R^2\backslash\{x\in\R^2: x_1<0,x_2=0\}$ w.r.t.~the intrinsic metric $\rho_E$ (note that $E$ is not quasi-convex).

Thus $f$ is intrinsically Lipschitz continuous with exception set $\Theta:=\{x\in\R^2: x_1<0,x_2=0\}$ in the sense of Definition \ref{def:pw-lip}.

The function $g\colon \R^2\longrightarrow\R$, $f(x)=\|x\|^2\arg(x)$ is only locally intrinsically Lipschitz continuous on $E$.
\end{example}

A classical method for proving Lipschitz continuity of a differentiable function also works for intrinsically Lipschitz continuity: 
\begin{example}
Let $A\subseteq \R^d$ open and let $f\colon A\to \R$ be differentiable
with $\sup_{x\in A}\|\nabla f(x) \|<\infty$. Then $f$ is intrinsically Lipschitz continuous on $A$ with Lipschitz constant $\sup_{x\in A}\|\nabla f(x) \|$. 
 
A proof can be found in  \cite[Lemma 3.6]{sz2016b}.
\end{example}

It is almost obvious, 
that a function $f\colon \R^2\to \R$, which is continuous and intrinsically Lipschitz continuous on 
$\R^2\setminus \{(x_1,x_2):x_1<0, x_2=0\}$, is Lipschitz continuous in $\R^2$.  
One can use that $\Theta:=\{(x_1,x_2):x_1<0, x_2=0\}$ does not pose a `hard' barrier, since every straight line 
connecting two points in $\R^2\setminus \Theta$ has at most one intersection point with $\Theta$ and so
one can conclude the Lipschitz continuity by approaching $\Theta$ from either side (we invite the reader to 
make this argument rigorous -- such a kind of argument will be used also in the proof of Theorem \ref{thm:main-th1}).   

To make the elementary property of `not being a hard barrier' precise, we define at this point 
the notion of permeability.  

\begin{definition}\label{def:permeable}
Let $E,\Theta\subseteq M$. 
\begin{enumerate}
\item The $\Theta$-{\em intrinsic metric} $\rho^\Theta_E$ on $E$ is defined by
\[
\rho^\Theta_E(x,y):=\inf\big\{\ell(\gamma)\colon \gamma\in \Gamma^\Theta(x,y)\big\} \,
\]
where $\Gamma^\Theta(x,y)$ is the set of all paths $\gamma\colon[a,b]\to M$ of finite length in $E$ from $x$ to $y$, such that $\overline{\{\gamma(t):t\in{[a,b]}\}\cap \Theta}$ is at most countable. 
(Again, we use the convention that $\inf\emptyset=\infty$.)
\item The $\Theta$-{\em finite intrinsic metric} $\rho^{\Theta,\textsc{fin}}_E$ on $E$ is defined by
\[
\rho^{\Theta,\textsc{fin}}_E(x,y):=\inf\big\{\ell(\gamma)\colon \gamma\in \Gamma^{\Theta,\textsc{fin}}(x,y)\big\} \,
\]
where $\Gamma^{\Theta,\textsc{fin}}(x,y)$ is the set of all paths $\gamma\colon[a,b]\to M$ of finite length in $E$ from $x$ to $y$, such that $\{\gamma(t):t\in{[a,b]}\}\cap \Theta$ is finite. 
\item We call $\Theta$ 
{\em permeable relative to $M$} iff $\rho_{M}=\rho_M^\Theta$.
\item We call $\Theta$ 
{\em finitely permeable relative to $M$} iff $\rho_{M}=\rho_M^{\Theta,\textsc{fin}}$.
\end{enumerate}
When the ambient space $(M,d)$ is understood and there is no danger of confusion, we simply
say $\Theta$ is (finitely) permeable.
\end{definition}

\begin{remark}
A set $\Theta\subseteq M$ is (finitely) permeable iff for any $x,y\in M$ and every $\varepsilon>0$ there exists a path
$\gamma$ from $x$ to $y$ in $M$ with $\ell(\gamma)<\rho_M(x,y)+\varepsilon$ and such that
\(
\overline{\{\gamma(t):t\in{[a,b]}\}\cap \Theta}
\) is at most countable (finite). Clearly, every finitely permeable set is permeable.

The notion of permeability is related to that of metrical removability \cite[Definition 1.1]{rajala2019}: A set $\Theta\in M$ is 
{\em metrically removable} if for all $x,y\in M$ and all $\varepsilon>0$ there exists a path $\gamma$ in $(M\setminus \Theta)\cup\{x,y\}$ from $x$ to $y$ with $\ell(\gamma)<d(x,y)+\varepsilon$. Since 
$d(x,y)\le \rho(x,y)$ it follows that every metrically removable set is finitely permeable.

Lemma 3.7 in \cite{rajala2019} states that if $M\subseteq\R^n$, then  a subset $\Theta\subseteq M$  is metrically removable if and only if $\rho_M=\rho_{M\setminus \Theta}$. Therefore, for subsets $M$ of the $\R^n$ with 
$\rho_M=d$  (i.e., $M$ is a length space), metrical removability corresponds to 
Definition \ref{def:permeable}, where `countable' or `finite' is replaced by `empty'.
\end{remark}

\begin{proposition}\label{prop:empty-interior} Let $\#M\ge 2$ and $\rho_M(x,y)<\infty$ for all $x,y\in M$.
If $\Theta\subseteq M$ is permeable, then it has no interior point with respect to the original metric $d$. 
\end{proposition}

\begin{proof}
Let $x$ be an interior point of $\Theta$ and $y\in M\setminus\{x\}$. 
By our assumption
there exists a path $\gamma\colon [0,1]\to M$ from $x$ to $y$ with finite length. By Lemma \ref{lem:injective-curve} there exists an arc $\eta\colon [0,1]\to \gamma([0,1])$ from $x$ to $y$ with $\ell(\eta)\le\ell(\gamma)$.
Since $x$ is an interior point of $\Theta$, there exists $r>0$ such that the ball $B_r(x):=\{y\in M: d(x,y)<r\}\subseteq \Theta$. By the continuity of $\eta$ there
exists $\delta>0$ such that $\eta([0,\delta])\subseteq B_r(x)\subseteq\Theta$. 
Thus $\eta([0,\delta])\subseteq\{\gamma(t):t\in{[0,1]}\}\cap \Theta$, such that the latter set is uncountable.
Since $\gamma$ was an arbitrary path in $M$ from $x$ to $y$, it follows that
 $\Gamma^\Theta(x,y)=\emptyset$,
and therefore $\rho_M^\Theta(x,y)=\infty>\rho_M(x,y)$.
\end{proof}

We will show in Section \ref{sec:smf} that all sufficiently regular sub-manifolds of the $\R^d$ of dimension smaller than $d$
are finitely permeable.
Next we show  that (finite) permeability transfers to subsets.

\begin{proposition}\label{prop:subset}
Let $\Theta_0\subseteq \Theta\subseteq M$. If $\Theta$ is (finitely) permeable, then $\Theta_0$ is (finitely) permeable.
\end{proposition}

\begin{proof}
Let $x,y\in M$ and $\varepsilon>0$. 
There exists $\gamma\colon [a,b]\to M$ such that 
$\ell(\gamma)<\rho_M(x,y)+\varepsilon$ and $\gamma([a,b])\cap \Theta$ has
countable closure. Since $\Theta_0\subseteq \Theta$, 
$\overline{\gamma([a,b])\cap \Theta_0}\subseteq \overline{\gamma([a,b])\cap \Theta}$. Therefore $\overline{\gamma([a,b])\cap \Theta_0}$ is countable, 
hence $\rho^{\Theta_0}_M(x,y)\le \rho_M(x,y)+\varepsilon$ from which the claim 
follows.

The `finitely permeable' case follows from similar considerations.
\end{proof}

\begin{proposition}\label{prop:subset-closed}
Let $\Theta_0\subseteq \Theta\subseteq M$. If $\Theta$ is (finitely) permeable relative to $(M,d)$ and $\Theta_0$ is closed in $M$, then $\Theta\setminus \Theta_0$ is (finitely) permeable relative to $(M\setminus \Theta_0,d)$.
\end{proposition}

\begin{proof} We only treat the permeable case, the finitely permeable one being almost identical.
 
Let $x,y\in M\setminus \Theta_0$ and let $\varepsilon>0$. If $\rho_{M\setminus \Theta_0}(x,y)=\infty$ there is nothing to show. Otherwise, there exists a path 
$\gamma\colon [0,1]\to M\setminus\Theta_0$ with $\gamma(0)=x$, $\gamma(1)=y$ and 
$\ell(\gamma)<\rho_{M\setminus\Theta_0}(x,y)+\frac{\varepsilon}{2}$. Since $\gamma([0,1])$ is compact and 
$\Theta_0$ is closed in $M$ there exists $\delta>0$ such that 
\[
\left\{z\in M\colon \inf_{t\in [0,1]}d\big(\gamma(t),z\big)<\delta\right\}\cap \Theta_0=\emptyset\,.
\]
Next we can find $0=t_0<t_1<\dots<t_n=1$ such that 
$\rho_{M\setminus \Theta_0}(\gamma(t_{k-1}),\gamma(t_k))<\frac{\delta}{2}$ for all $k=1,\ldots,n$. Since $\Theta$ is permeable
relative to $M$, there exist $\eta_1,\ldots,\eta_n\colon [0,1]\to M$ with
$\eta_k(0)=\gamma(t_{k-1})$, $\eta_k(1)=\gamma(t_k)$, 
\begin{align*}
\ell(\eta_k) 
&< \rho_M(\eta_k(0),\eta_k(1))+\min\left(\frac{\varepsilon}{2n},\frac{\delta}{2}\right)\\  
\end{align*}
and $\eta_k([0,1])\cap \Theta$ has countable closure for every $k\in\{1,\dots,n\}$. Now 
\begin{align*}
\ell(\eta_k) &< \rho_{M\setminus \Theta_0}(\gamma(t_{k-1}),\gamma(t_k))+\min\left(\frac{\varepsilon}{3n},\frac{\delta}{2}\right)<\delta\,,
\end{align*}
so that if $t\in[0,1]$,
$d(\eta_k(t),\eta_k(0))\le \ell(\eta_k)< \delta$, and therefore $\eta_k(t)\notin \Theta_0$. Therefore 
$\eta\colon [0,1]\to M$, the concatenation of the paths $\eta_1,\ldots,\eta_n$ is a path in $M\setminus \Theta_0$ with 
$\ell(\eta)<\ell(\gamma)+\frac{\varepsilon}{2}<\rho_{M\setminus\Theta_0}(x,y)+\varepsilon$ such that the closure of $\eta([0,1])\cap (\Theta\setminus\Theta_0)$ is countable.
\end{proof}

We now state our first main result.

\begin{theorem}\label{thm:main-th1}
Let $\Theta\subseteq M$ be permeable, $(\msr,d_\msr)$ a  metric space. Then every continuous function $f\colon M\to \msr$,
which is intrinsically $L$-Lipschitz continuous on $E=M\setminus \Theta$, is intrinsically $L$-Lipschitz continuous on the whole of $M$.
\end{theorem}

For the proof of this result we use the following classical theorem, see for example \cite[Theorem 6.11]{kechris1995classical}: 

\begin{theorem}[Cantor-Bendixson]\label{thm:cantor-bendixson}
Let $M$ be polish. 
For every closed $A\subseteq M$ denote by $I(A)$ are the isolated points in $A$ and
$H(A):=A\setminus I(A)$.
For every ordinal $\alpha$ we set
\[
H^{\alpha}(A):=
\begin{cases}
A&\text{if $\alpha=0$}\,,\\
H\big(H^{\beta}(A)\big)\,, & \text{if $\alpha$  is the successor of $\beta$} \,,\\
\bigcap_{\beta<\alpha} H^\beta(A)\,,& \text{if $\alpha$  is a limit ordinal}\,.
\end{cases}
\]
Then for such a closed $A\subseteq M$ there exists a countable ordinal $\alpha_0$  such that for all  $\alpha\geq \alpha_0:H^{\alpha}(A)=H^{\alpha_0}(A)$, i.e.,
$H^{\alpha_0}(A)$ is a perfect set. The smallest such ordinal $\alpha_0$
is called the {\em Cantor-Bendixson rank} of $A$.

In particular, $H^{\alpha_0}(A)=\emptyset$ for some countable ordinal $\alpha_0$ iff $A$ 
is countable.
\end{theorem}

%
%
%

\begin{proof}[{\bf Proof of Theorem \ref{thm:main-th1}}] 
Let  $f\colon M\to \msr$  be a continuous function  
which is intrinsically Lipschitz continuous on
$E:=M\setminus\Theta$. Denote by $L$ the Lipschitz constant of $f$. Let $x,y$ in $M$ and let $\varepsilon>0$. 

Since $\Theta$ is permeable, there exists a path $\gamma\colon [a,b]\to M$ from $x$ to $y$ with $\ell(\gamma)<\rho_M(x,y)+\varepsilon$
and such that 
$A_\gamma:=\overline{\{\gamma(t)\colon t\in [a,b]\}\cap\Theta}$ is countable. 

Invoking Lemma \ref{lem:injective-curve} we may assume that $\gamma$ is injective. Furthermore we may and will assume that $\forall t\in [a,b]\colon \ell(\gamma|_{[0,t]})=t$ and, in particular, $a=0,b=\ell(\gamma)$.

If we can show that the map $f\circ \gamma\colon [0,b]\to Y$ is $L$-Lipschitz continuous, then we are done, 
because then 
\[
d_Y\big(f(x),f(y)\big)=d_Y\big(f\circ\gamma(0),f\circ \gamma(b)\big)\le L b=L\ell(\gamma)<L(\rho_M(x,y)+\varepsilon)\,.
\]
Let $A:=\gamma^{-1}(A_\gamma)\subseteq [0,b]$. Then $A$ is closed since $\gamma$ is continuous. We start by showing that if $C_0$ is a 
connected component of $[0,b]\setminus A$ (clearly, $C_0$ is an interval with non-empty interior), then $f\circ\gamma$ is $L$-Lipschitz
on $\overline{C_0}$. Indeed, let $r,s\in C_0^\circ$ with $r<s$. 
Since $[r,s]\cap A= \emptyset$, the restricted arc $\gamma|_{[r,s]}$ is an arc in $E$ and 
therefore $\rho_E(\gamma(r),\gamma(s))\le \ell(\gamma|_{[r,s]})= (s-r)$\,. Since $f$ is intrinsically
$L$-Lipschitz continuous on $E$, 
\begin{equation}\label{eq:lip-0}
d_Y\big(f\circ\gamma(r),f\circ \gamma(s)\big)\le L\rho_E(\gamma(r),\gamma(s))\le L(s-r)\,.
\end{equation}
By continuity of $f$, Equation \eqref{eq:lip-0} also holds for $r,s\in \overline{C_0}$.

Next let $C_1$ be a connected component of $[0,b]\setminus H(A)$. If $r,s\in C_1^\circ$ with $r<s$, then, since $[r,s]$ does not contain an accumulation point of $A$,  
there exist $t_1,\ldots,t_n\in [0,b]$ such that $t_1<\ldots<t_n$ and
$[r,s]\cap A=\{t_1,\ldots,t_n\}$. Set $t_0:=r$, $t_{n+1}:=s$. Then for every $k\in\{1,\dotsc,n+1\}$ the interval 
$(t_{k-1}, t_{k})$ does not contain a point of $A$, so
\[
d_Y\big(f\circ\gamma(t_{k-1}),f\circ \gamma(t_k)\big)\le L(t_k-t_{k-1})\,,
\]
and therefore
\begin{align}
\nonumber d_\msr\Big(f\circ\gamma(r),f\circ\gamma(s)\Big)
&\le \sum_{j=1}^{n+1} d_\msr\Big(f\circ\gamma(t_{j-1}),f\circ \gamma(t_{j})\Big) \\
\label{eq:lip-metrik}&\le L \sum_{j=1}^{n+1} \big(t_j-t_{j-1}\big)= L (s-r)\,.
\end{align}
That is, 
\begin{equation}\label{eq:lip-1}
d_Y\big(f\circ\gamma(r),f\circ \gamma(s)\big)\le L(s-r)\,.
\end{equation}
and by continuity of $f$, Equation \eqref{eq:lip-1} also holds for $r,s\in \overline{C_1}$. 

We proceed by a transfinite induction argument, where the base case has already been dealt with.
The induction hypothesis is that for an ordinal $\alpha$, all ordinals $\beta<\alpha$ and
every connected component $C_\beta$ of $[0,b]\setminus H^\beta(A)$ we have that 
 $f\circ \gamma$ is $L$-Lipschitz continuous on 
$\overline{C_\beta}$. 
In order to perform the induction step, we need to show that this property extends to $\alpha$, that is,
 for every connected component
$C_\alpha$ of $[0,b]\setminus H^\alpha(A)$, $f\circ \gamma$ is $L$-Lipschitz continuous on $\overline{C_\alpha}$. 

If $\alpha$ is not a limit ordinal, we can use the same argument
as in the step from the ordinal 0 to the ordinal 1.

Now assume, that $\alpha$ is a limit ordinal. Let $r,s\in C_\alpha^\circ$, $r<s$. 
There exists an ordinal $\beta_0<\alpha$ such that $[r,s]\cap H^{\beta_0}(A)=\emptyset$ :
Otherwise, there is an increasing sequence $(\beta_n)_{n\in \N}$, $\beta_n<\alpha$ for all $n$,  
and a sequence $(t_n)_{n\in \N}$, $t_n\in H^{\beta_n}(A)\cap[r,s]$   with $\lim_{n\to\infty} t_n\in H^\alpha(A)$.
But this is impossible, since $[r,s]\subseteq C_\alpha^\circ$ and $C_\alpha^\circ\cap H^\alpha(A)=\emptyset$, so that  $\inf\{|u-v|:u\in H^\alpha(A),v\in [r,s])>0$. 

But $[r,s]\cap H^{\beta_0}(A)=\emptyset$  implies that $[r,s]\subseteq \overline{C_{\beta_0}}$ for some 
connected component $C_{\beta_0}$ of $[0,b]\setminus H^{\beta_0}(A)$.
By the induction hypothesis, $f\circ\gamma$ is $L$-Lipschitz on $\overline{C_{\beta_0}}$, and we get Equation~\eqref{eq:lip-1}.

We are ready to finish the proof. By Theorem \ref{thm:cantor-bendixson} there exists a countable
ordinal $\alpha_0$ such that $H^{\alpha_0}(A)=\emptyset$. But then $[0,b]\setminus H^{\alpha_0}(A)=[0,b]$
is the only connected component
and therefore $f\circ \gamma$ is $L$-Lipschitz continuous on $[0,b]$.
\end{proof}

\begin{remark}
Note that the inequalities in \eqref{eq:lip-metrik}
cannot be generalized in a straightforward way to a related notion of ``intrinsically Hölder continuous''.
\end{remark}

\begin{remark}\begin{enumerate} \item Theorem \ref{thm:main-th1}
generalizes Lemma 3.6 in \cite{sz2016b}. In the latter it is
assumed that the exception set is a finitely permeable submanifold of $\R^d$.

\item Theorem \ref{thm:main-th1} and its proof should also be compared with the results 
\cite[Theorem 2.5 and Proposition 2.2]{craig-et-al-2015}, which together imply the following:
Let $I$ be a real interval, $f : I\to\R$ a function and $E\subseteq I$. If
\begin{itemize}
\item  $E$ has no perfect subsets,
\item $f : I\to\R$ is continuous,
\item  the pointwise Lipschitz constant of $f$ is bounded by a constant $C$ at
every point of $I\setminus E$.
\end{itemize}
Then $f$ is $C$-Lipschitz on $I$.

This result is obviously also related to our 1-dimensional permeability criterion, Theorem \ref{thm:main-1d}, below. 

\end{enumerate}
\end{remark}

We have two immediate corollaries of Theorem \ref{thm:main-th1}:

\begin{corollary}
Let $M$ be a $C$-quasi-convex space and let $\Theta\subseteq M$ be permeable. Then every continuous function $f\colon M\to \msr$ 
into a  metric space $(\msr,d_\msr)$,
which is intrinsically $L$-Lipschitz continuous on $E=M\setminus \Theta$, is $CL$-Lipschitz continuous on the whole of $M$ (i.e., with respect to $d$).
\end{corollary}

\begin{corollary}\label{cor:main-th1-length}
Let $M$ be a length space and let $\Theta\subseteq M$ be permeable. Then every continuous function $f\colon M\to \msr$ 
into a  metric space $(\msr,d_\msr)$,
which is intrinsically $L$-Lipschitz continuous on $E=M\setminus \Theta$, is $L$-Lipschitz continuous on the whole of $M$ (i.e., with respect to $d$).
\end{corollary}

For example, Corollary \ref{cor:main-th1-length} can be applied for $M=\R^d$ with the euclidean metric, which is
a length space. We will take a deeper look at this example in Section \ref{sec:smf}.

The next result states that, if a function is intrinsically Lipschitz except on a closed permeable set, then the
Lipschitz constant does not change when one enlarges the exception set to another permeable set.

\begin{proposition}\label{prop:equalL}
Let $\Theta_0\subseteq \Theta\subseteq M$ with $\Theta$ permeable and $\Theta_0$ closed. Let 
$f\colon M\to Y$ be intrinsically Lipschitz  except on $\Theta_0$.
Then  $f$ is intrinsically Lipschitz  except on $\Theta$ and
\[
\sup_{x,y\in M\setminus \Theta}\frac{d_Y(f(x),f(y))}{\rho_{M\setminus \Theta}(x,y)} 
=\sup_{x,y\in M\setminus \Theta_0}\frac{d_Y(f(x),f(y))}{\rho_{M\setminus \Theta_0}(x,y)}
\]
\end{proposition}

\begin{proof}
Let $N=M\setminus \Theta_0$. By assumption, $f$ is intrinsically
Lipschitz continuous on $N$. 
As $\Theta\supseteq \Theta_0$,  $f$ is also intrinsically Lipschitz continuous on $M\setminus\Theta=N\setminus(\Theta\setminus\Theta_0)$, and  $\Theta\setminus\Theta_0$ is 
permeable in $N$, by Proposition \ref{prop:subset-closed}. Therefore the assertion follows from 
Theorem \ref{thm:main-th1}.
\end{proof}

\begin{example}
Consider $(M,d):=(\R^2,|.|)$ and $\Theta:=\Q^2$. Then 
$\rho^{\Q^2}_{\R^2}(x,y)=|x-y|$ for all $x,y\in\R^2$.
By Corollary \ref{cor:main-th1-length}, every intrinsically Lipschitz 
function on $M$ with exception set $\Q^2$, which is continuous on $\R^2$,
is Lipschitz on the whole of $\R^2$.
\end{example}


We now
show that in the one-dimensional euclidean
case the permeable sets $\Theta$ are precisely those 
for which every function is Lipschitz iff it is continuous and intrinsically Lipschitz with 
exception set $\Theta$. Note that for subsets of $\R$ permeability is equivalent to having countable closure.

\begin{theorem}\label{thm:main-1d}
Let $\Theta\subseteq \R$. Then $\Theta$ has countable closure if and only if
for all intervals $I\subseteq \R$ and all functions $f\colon I\to \R$ 
the properties
\begin{itemize}
\item $f$ is intrinsically Lipschitz continuous with exception set $\Theta$,
\item $f$ is continuous,
\end{itemize}
imply that $f$ is Lipschitz continuous on $I$.
\end{theorem}

\begin{proof}
The only if part follows from Corollary \ref{cor:main-th1-length}.

\smallskip

For the ``if part'' assume to the contrary that  $\Theta$ has uncountable 
closure. Then $\overline{\Theta}=A\cup P$, where $A$ is countable and 
$P$ is perfect by the Cantor-Bendixson theorem, \cite[Theorem 6.4]{kechris1995classical}. 
Hence, $P$ contains a homeomorphic copy of the Cantor set by \cite[Theorem 6.5]{kechris1995classical}. Let $f$ be the corresponding
Cantor staircase function. Then $f$ is continuous and $f$ is constant on
every connected component of $\R\setminus \Theta$.
But the Cantor staircase function is not 
Lipschitz.
\end{proof}

%

The following proposition by Tapio Rajala\footnote{University of Jyvaskyla, Finland} (personal communication) connects consequences of permeability and the statement of Corollary \ref{cor:main-th1-length}, that intrinsically Lipschitz continuity for a given exception sets implies the overall Lipschitz continuity. It should also be compared to Proposition \ref{prop:subset}, which states that subsets of permeable sets are permeable.

\begin{proposition}\label{prop:subset-cont}
Let $\Theta\subseteq M$ have the property that every function 
$f\colon M \to Y$ which 
is continuous and intrinsically ($L$-)Lipschitz continuous with exception set $\Theta$
is ($L$-)Lipschitz continuous. 

Then every subset $\Theta_0$ of $\Theta$ enjoys the same property.
\end{proposition}

\begin{proof}
If a function $f\colon M\to Y$ is ($L$-)Lipschitz continuous with respect to $\rho_{M\setminus\Theta_0}$, then, as $\rho_{M\setminus\Theta_0}\leq \rho_{M\setminus\Theta}$, $f$ is ($L$-)Lipschitz continuous with respect to $\rho_{M\setminus\Theta}$ implying ($L$-)Lipschitz continuity of $f$ on the whole of $M$.
\end{proof}

\begin{remark} \label{rem:question-Rd} With regard to  Corollary \ref{cor:main-th1-length}, Theorem \ref{thm:main-1d} and Proposition \ref{prop:subset}, one may ask the following interesting question:

\medskip

\textit{In $\R^d$, are the permeable subsets $\Theta$ precisely those for which every function is $L$-Lipschitz iff it is continuous and intrinsically $L$-Lipschitz with 
exception set $\Theta$?}

\medskip

The following example, proposed by Tapio Rajala (personal communication), shows that one may not reduce `$L$-Lipschitz' to merely `Lipschitz' in the above question:
Consider the set $\Theta=([0,1]\setminus\mathbb{Q})^2\subseteq M=[0,1]^2$. Then the intrinsic metric $\rho_{M\setminus\Theta}$ is given by the $1$-distance, see \cite[Proposition 3.6]{rajala2019}. Hence, if a function is continuous and $L$-Lipschitz with respect to $\rho_{M\setminus\Theta}$, then it will be continuous and $\sqrt{2}L$-Lipschitz with respect to the euclidean metric on $[0,1]^2$, and therefore Lipschitz. On the other hand, $\Theta$ is not permeable, which we show in the subsequent proposition. \end{remark}

\begin{proposition}
Let $M=[0,1]^2$ endowed with the euclidean distance and let 
$\Theta=([0,1]\setminus\mathbb{Q})^2$.
Then $\Theta$ is not permeable.
\end{proposition}

\begin{proof}
Assume that $\gamma$ is a curve from $(0,0)$ to $(1,1)$ of length $\ell(\gamma)<\sqrt{2}+\varepsilon$. As we can always shorten parts of $\gamma$, where its first component is not monotonically increasing, by a vertical line with rational first component, it suffices to assume $\gamma$ to be given by a function $\tilde \gamma\colon[0,1]\to[0,1]$ of arc length smaller than $\sqrt{2}+\varepsilon$. Similarly, one may replace intervals, where $\tilde\gamma$ is strictly decreasing, by a horizontal segment with rational function value. So we may assume $\tilde\gamma$ to be monotonically increasing. As monotonically increasing function, its derivative exists almost everywhere in $[0,1]$. The derivative can not be zero almost everywhere, for then the function's arc length were $2$ which is not smaller than $\sqrt{2}+\varepsilon$ for arbitrary $\varepsilon$, see e.g. \cite[Theorem 4]{delegl14}. Thus there must be a set $A$ with $\lambda(A)>0$, where $\tilde{\gamma}'(x)>0$ for all $x\in A$. Since the derivatives are positive on $A$, it follows that $\tilde{\gamma}(x)<\tilde{\gamma}(x')$ for $x<x'\in A$. Since $M\setminus\Theta = \left(([0,1]\cap\mathbb{Q})\times([0,1]\setminus\mathbb{Q})\right)\cup\left(([0,1]\setminus\mathbb{Q})\times([0,1]\cap\mathbb{Q})\right)$, \begin{align}\label{eq:unionmt}
\mathrm{Graph}(\tilde{\gamma}\!\!\mid_A)\cap (M\setminus\Theta)=\bigcup_{x\in A\cap\mathbb{Q}}\left\{(x,\tilde{\gamma}(x))\right\}\cup\bigcup_{y\in \mathbb{Q}\cap\tilde\gamma(A)}\left\{\big((\tilde{\gamma}\!\!\mid_A)^{-1}(y),y\big)\right\}.
\end{align}
As $A$ is uncountable, $\mathrm{Graph}(\tilde{\gamma}\!\!\mid_A)$ has uncountably many values in the first component, but $\mathrm{Graph}(\tilde{\gamma}\!\!\mid_A)$ intersects the first union of \eqref{eq:unionmt} in countably many points. In the same way, the second union intersects $\mathrm{Graph}(\tilde{\gamma}\!\!\mid_A)$ only in countably many points. Therefore $\mathrm{Graph}(\tilde{\gamma}\!\!\mid_A)\subseteq \mathrm{Graph}(\tilde{\gamma})$ intersects $\Theta$ in uncountably many points. We conclude that $\Theta$ is not permeable. 
\end{proof}

We have shown in Proposition \ref{prop:empty-interior} that permeable sets cannot contain 
interior points.
It therefore makes sense to study subsets $\Theta \subseteq \R^d$
that have no interior points relative to $\R^d$, and
in the next section we further specialize to submanifolds of $\R^d$
with dimension strictly smaller than $d$.

\section{Sub-manifolds of $\R^d$ as permeable sets}\label{sec:smf}

\begin{definition}\label{defi:submanifold}
Let $d,m,k\in \N\cup\{0\}$, $m < d$,  and let $\Theta\subseteq \R^d$.
We say $\Theta$ is an $m$-dimensional $C^k$-{\em submanifold} of $\R^d$ iff 
for every $\xi\in \Theta$ there exist open sets $U,V\subseteq \R^d$ 
and a $C^k$-diffeomorphism $\Psi:V\to U$ such that $\xi\in U$ and
for all $y=(y_1,\dots,y_d)\in V$ it holds
$\Psi(y)\in \Theta \Longleftrightarrow y_{m+1}=\dots=y_d=0$. 
In the case where $k=0$, by a $C^0$-diffeomorphism we mean a homeomorphism,
and we also call $\Theta$  a {\em topological submanifold (\textsc{top}-submanifold) }. 
\end{definition}

\begin{definition}
A \textsc{top}-submanifold of $\R^d$ of dimension $m<d$ is called {\em Lipschitz} or of {\em class \textsc{lip}} if the 
mappings $\Psi, \Psi^{-1}$ from Definition \ref{defi:submanifold} are Lipschitz continuous  on every compact
subset of their respective domain.
\end{definition}

\begin{corollary}
The class of Lipschitz submanifolds contains those that possess
continuously differentiable mappings $\Psi,\Psi^{-1}$ (class $C^1$) as well
as those having mappings $\Psi,\Psi^{-1}$ that are continuous and piecewise
linear, resp.~piecewise differentiable, on subdivisions of $U,V$ into polyhedra (class \textsc{pl}, resp.~class \textsc{pdiff}).
\end{corollary}

Before we state the main result of this section, we prove a preparatory lemma.

\begin{lemma}\label{lemma:polygonal-chain}
Let $\Theta\subseteq\R^d$ be a Lebesgue-nullset.
Then for all $x,y\in \R^d$ and all $\varepsilon>0$ there exists a 
polygonal chain $\gamma\colon [a,b]\to \R^d$ such that 
$\ell(\gamma)<\|y-x\|+\varepsilon$ and 
$\{t\in [a,b]\colon \gamma(t)\in \Theta\}$ is a Lebesgue-nullset.
\end{lemma}

\begin{proof}
For the case $d=1$ there is nothing to prove. 
If $d>1$, consider the $(d-1)$-dimensional ball $B$ with center 
$\frac{1}{2}(x+y)$
and radius $\frac{1}{2}\sqrt{(\|y-x\|+\frac{\varepsilon}{2})^2-\|y-x\|^2}$
that lies in the hyperplane orthogonal to $x-y$ and passes through 
$\frac{1}{2}(x+y)$. Then the convex hull of $B\cup\{x,y\}$, which we denote by $C$, is a $d$-dimensional
double cone and, since $\Theta$ is a Lebesgue-nullset,
we have $\lambda^d(C\cap \Theta)=0$, where $\lambda^d$ is the Lebesgue-measure on $\R^d$.

By Fubini's theorem,
\begin{align*}
0&=\lambda^d(C\cap \Theta)\\
&=\frac{1}{2}\|x-y\|\int_B \int_{[0,1]} \Big(1_\Theta((1-t)x+t z)+1_\Theta((1-t)y+t z)\Big) (1-t)^{d-1} dt\, dz\,.
\end{align*}
From this we conclude that 
\[
\int_0^1 \Big(1_\Theta((1-t)x+t z)+1_\Theta((1-t)y+t z)\Big) (1-t)^{d-1} dt
=0\,,
\]
for almost all $z\in B$. We may
choose one such $z\in B$, and this we have $1_\Theta((1-t)x+t
z)+1_\Theta((1-t)y+t z)=0$ for almost all $t\in(0,1)$. Thus the proof is finished. 
\end{proof}

\begin{theorem}\label{th:lip}
Let $\Theta\subseteq\R^d$ be a \textsc{lip}-submanifold  which in addition 
is  a closed subset of $\R^d$.
Then for all $x,y\in \R^d$ and all $\varepsilon>0$ there exists a path 
$\gamma\colon[a,b]\to\R^d$ from $x$ to $y$ with
$\ell(\gamma)<\|x-y\|+\varepsilon$ and such that 
$\gamma\big([a,b]\big)\cap\Theta$ is finite. Therefore, $\Theta$ is finitely permeable and hence permeable.
\end{theorem}

\begin{proof} Let $x,y\in \R^d$ and $\varepsilon>0$.\\
\textit{Step 1:} Note that, since $\Theta$ is a Lipschitz topological submanifold and is
thus locally the Lipschitz image of a Lebesgue-nullset, it is itself a
Lebesgue-nullset (with respect to $\R^d$) as bi-Lipschitz homeomorphisms pertain measurability of sets and therefore
preserve Lebesgue-nullsets, see, e.g.,\cite[Lemma 7.25]{rudin86real}.
By virtue of Lemma \ref{lemma:polygonal-chain} we may thus restrict our
considerations to the case where 
$F:=\big\{t\in [0,1]\colon (1-t)x+ty\in \Theta\big\}$ has Lebesgue measure 0
(in $[0,1]$).

If $F$ is finite, then we are done. Otherwise assume first that $x\notin \Theta$. 
Since $\Theta$ is a closed set, we can find $z\in \Theta$
such that the line segment $\gamma_1$ connecting $x$ and $z$ intersects $\Theta$
precisely in $z$. If we can find a path $\gamma_2\colon[a,b]\to\R^d$ 
from $z$ to $y$ with 
$\ell(\gamma_2)<\|y-z\|+\varepsilon$ and such that 
$\gamma_2\big([a,b]\big)\cap\Theta$ is finite, then the concatenation 
$\gamma$ of the paths $\gamma_1$ and $\gamma_2$
is a path with the required properties and we are done. Thus we may assume that
$x\in \Theta$ and, by the same argument, that $y\in \Theta$.\smallskip

\noindent\textit{Step 2:}  
Write $g\colon [0,1]\to \R^d$, $g(t):=(1-t)x+ty$.
Since $\big\{ g(t) \colon t\in F\big\}\subseteq\Theta$ is compact, 
there exist finitely many $t_1,\ldots,t_n\in F$ and bounded open environments 
$U_j$ of $g(t_j)$, $V_j$ of $0$ and $\Psi_j\colon V_j\to U_j$
bi-Lipschitz such that for all
$z=(z_1,\dots,z_d)\in V_j$ it holds $\Psi_j(z)\in \Theta \Longleftrightarrow
z_{m+1}=\dots=z_d=0$. 

$F$ is a non-empty closed subset of $[0,1]$, so we can 
write
\[
[0,1]\setminus F=\bigcup_{k=1}^\infty (a_k,b_k)\,,
\]
where the right hand side is a disjoint union and, since $\lambda(F)=0$,
$\sum_{k=1}^\infty (b_k-a_k)=1$. 
Now for every $K\in \N$ there exist $c_0,\ldots,c_K,d_0,\ldots,d_K$ with 
\begin{equation}\label{eq:decomposition}
\Big(\bigcup_{k=1}^K (a_k,b_k)\Big)^c
=\bigcup_{k=0}^{K} [c_k,d_k] \supseteq F\,,
\end{equation}
where $[c_0,d_0],\ldots,[c_K,d_K]$ are again disjoint. 
We may assume that $K$ is large enough to guarantee that
for every interval $[c_k,d_k]$ there exists $j\in \{1,\ldots,n\}$ such that
$g( [c_k,d_k])\subseteq U_j$. In addition, we may assume that for every $j$ the functions 
$\Psi_j$ and $\Psi^{-1}_j$ are Lipschitz with common constant $L_j$. If we can find, for every $k=0,\ldots,K$,
a path $\gamma_k\colon[0,1]\to \R^d$ from $g(c_k)$ to $g(d_k)$ with 
$\ell(\gamma_k)<\|g(d_k)-g(c_k)\|+\frac{\varepsilon}{K+1}$ and such that 
$\gamma_k([c_k,d_k])\cap\Theta$ is finite, then we can construct a
path with the required properties. We may therefore concentrate 
on the case where the whole of $g([0,1])$ is contained in a single $U_j$, which we will do in Step 3.
\smallskip

\noindent\textit{Step 3:} Write $U:=U_j$, $V:=V_j$, $\Psi:=\Psi_j$, $L:=L_j$.  
Since $\Psi$ and $\Psi^{-1}$ are Lipschitz continuous with constant $L$, 
we have for every finite collection of paths
$\eta_1,\ldots,\eta_N$ in $U$
\[
 \sum_{n=1}^N \ell(\Psi^{-1}\circ\eta_n)\le L\sum_{n=1}^N \ell(\eta_n) \,.
\]

for any finite collection of paths
$\kappa_1,\ldots,\kappa_N$ in $V$
\[
\sum_{n=1}^N \ell(\Psi\circ\kappa_n)\le L\sum_{n=1}^N \ell(\kappa_n)\,.
\]

We repeat the earlier argument
to get, for every $K\in\N$, a disjoint union of the type \eqref{eq:decomposition}.
Now we choose $K$ big enough to ensure $\sum_{k=0}^K (d_k-c_k)<\frac{\varepsilon}{2 L^2 \|y-x\|}  $.
We write
$x_k:=(1-c_k)x+c_k y$ and  
$y_k:=(1-d_k)x+d_k y$.
For every $k\in\{0,\ldots,K\}$, let $\kappa_k\colon [c_k,d_k]\to U$, 
$\kappa_k(t):=(1-t)x+t y$ that is, $\kappa_k$ is a parametrization of the 
line-segment
from $x_k$ to $y_k$. Now $\kappa_0,\ldots,\kappa_K$ is a finite collection of 
 paths with 
\begin{align*}
\lefteqn{\sum_{k=0}^K \|\Psi^{-1}\circ\kappa_k(d_k)-\Psi^{-1}\circ\kappa_k(c_k)\|}\\
&\le \sum_{k=0}^K \ell(\Psi^{-1}\circ\kappa_k)\le L\sum_{k=0}^K \ell(\kappa_k)\\
&\le L\sum_{k=0}^K (d_k-c_k)\|y-x\|<\frac{\varepsilon}{2 L}\,.
\end{align*}
Set $\ell_k:=\|\Psi^{-1}\circ\kappa_k(d_k)-\Psi^{-1}\circ\kappa_k(c_k)\|$. 
For every $k$ with $\ell_k=0$ let $\eta_k$ be constant equal to 
$\Psi^{-1}\circ\kappa_k(c_k)$.

Denote by $e_d$ the vector $(0,\ldots,0,1)$.
For every $k$ with $\ell_k>0$ 
we construct a path $\eta_k\colon [0,2 \ell_k]\to V$ by
\[
\eta_k(t)=\begin{cases}
\Psi^{-1}\circ\kappa_k(c_k)+t a e_d& \text{if } 0\le t \le \frac{\ell_k}{2}\,,\\
\frac{3\ell_k-2t}{2\ell_k}\Psi^{-1}\circ\kappa_k(c_k)+\frac{2t- \ell_k}{2\ell_k}\Psi^{-1}\circ\kappa_k(d_k)+\frac{\ell_k}{2} a e_d&  \text{if } \frac{\ell_k}{2}\le t \le \frac{3\ell_k}{2}\,,\\
\Psi^{-1}\circ\kappa_k(d_k)+(2\ell_k-t)a  e_d& \text{if }  \frac{3\ell_k}{2}\le t \le 2 \ell_k\,,
\end{cases}
\]
where $a\in (0,1)$ is small enough so that $\eta_k(t)\in V$ for all $t\in [0,2 \ell_k]$.
By construction, $\ell(\eta_k)\le 2\ell_k$, such that 

\[
\sum_{k=0}^K \ell(\Psi\circ\eta_k)\le L\sum_{k=0}^K \ell(\eta_k)\le 2L\sum_{k=0}^K \ell_k< \varepsilon.
\]

Now define $\gamma$ as the concatenation of the following paths: 
\begin{itemize}
\item the paths $\Psi\circ\eta_k$ from $x_k$ to $y_k$ for $k=0,\ldots,K$;
 \item the line segments from $y_{k-1}$ to $x_k$ with 
lengths $\hat\ell_k:=\|x_k-y_{k-1}\|$ for $k=1,\ldots,K$.
\end{itemize}
Summing up, we get for the length of $\gamma$
\begin{align*}
\ell(\gamma)
&=\sum_{k=1}^K \hat \ell_k+\sum_{k=0}^K \ell(\Psi\circ\eta_k)\\
&=\sum_{k=1}^K \|x_k-y_{k-1}\|+\sum_{k=0}^K \ell(\Psi\circ\eta_k)\\
&<\|y-x\|+\varepsilon\,,
\end{align*}
and $\gamma$ has only finitely many intersections with $\Theta$. 
\end{proof}

\begin{corollary}
Let $\Theta$ be a Lipschitz submanifold of the $\R^d$ which is closed
as a subset. 

Then every continuous function $f\colon \R^d\to \R$ 
which is intrinsically Lipschitz continuous with exception set 
$\Theta$ is Lipschitz continuous.
\end{corollary}

\begin{proof}
This follows immediately from Corollary \ref{cor:main-th1-length} and Theorem \ref{th:lip}.
\end{proof}

\begin{example}
Consider the topologist's sine: 
\begin{align*}
\Theta&:=\left\{\left(t,\sin\left(\tfrac{1}{t}\right)\right)\colon t\in (0,\infty)\right\}\,,\\
\overline{\Theta}&\;=\left\{\left(t,\sin\left(\tfrac{1}{t}\right)\right)\colon t\in (0,\infty)\right\}\cup\big\{(0,s)\colon s\in [-1,1]\big\}
\,.
\end{align*}
It is readily checked that $\rho^{\overline{\Theta}}_{\R^2}=\rho^\Theta_{\R^2}=\rho_{\R^2}$, i.e., both $\Theta$ and
$\overline{\Theta}$ are permeable.  
$\Theta$ is a sub-manifold of $\R^2$, which is not a subset of a topologically closed submanifold of the $\R^d$,
while $\overline{\Theta}$ is closed, but not a submanifold.
Note also, that  for $x=(0,0)$, $y=(1,0)$ there is no path connecting $x$ and
$y$ of length smaller than 2 which has a finite intersection with
$\Theta$, in contrast to the case of closed Lipschitz submanifolds.
Therefore $\Theta$ is permeable but not finitely permeable.
\end{example}

The following example shows that one cannot simply dispense with the assumption that 
the exception set is topologically closed.

\begin{example}
Recall that the classical Cantor  set $C$ is the topological closure of the set
\[
C_0=\left\{\sum_{k=1}^n d_k3^{-k}: n\in \N,  d_k\in \{0,2\}\right\}\,.
\]
Every element from $C_0$ is the limit of an increasing sequence in $\R\setminus \overline{C_0}$:
\begin{align*}
0&=\lim_{j} -3^{-j} \\
\sum_{k=1}^n d_k3^{-k}&=\lim_{j\to\infty} \left(\sum_{k=1}^n d_k3^{-k}-3^{-(j+n)}\right)      &  (n\in \N, d_n=2)
\end{align*}
Denote by $D_0$ the union of all elements of all these sequences. Then $D_0$ consists only of 
isolated points and $\overline{D_0}\supseteq \overline {C_0}=C$. Therefore $D_0$ is a 0-dimensional submanifold of $\R$ which is not permeable. 

By extruding $D_0$ to $D_{d-1}:=D_0\times \R^{d-1}$ we get an example of a $(d-1)$-dimensional 
$C^\infty$-manifold which is not permeable (and not topologically closed).
\end{example}

We conclude this section with more examples of permeable sets.
First note that we can somewhat relax the requirement that a 
Lipschitz manifold be a closed subset of $\R^d$, since by 
Proposition \ref{prop:subset} the property $\rho^\Theta_{\R^d}=\rho_{\R^d}$
extends to subsets.

\smallskip

The conclusion of Theorem \ref{th:lip} also holds for unions of closed
Lipschitz submanifolds with transversal intersections, with the following
notion of {\em transversal intersection}:

\begin{definition}
Let $\Theta_1$ and $\Theta_2$ be two $C^k$-submanifolds of the $\R^d$,
and let $\xi\in
\Theta_1\cap\Theta_2$.
We say that $\Theta_1$ and $\Theta_2$ {\em intersect transversally} in $\xi$,
iff 
there exist 
\begin{itemize}
\item open sets $U,V\subseteq \R^d$ such that $\xi\in U$
\item  a $C^k$-diffeomorphism $\Psi:V\to U$  
\item linear subspaces $E_1$, $E_2$ with $\dim(E_j)=\dim(\Theta_j)$, $j=1,2$,
\end{itemize}
such that 
$\Psi(V\cap E_j)=U\cap \Theta_j$, $j=1,2$.

We say $\Theta_1$ and $\Theta_2$ {\em intersect transversally} iff they intersect
transversally in every $\xi\in  \Theta_1\cap\Theta_2$.

In the cases of a \textsc{top}- or \textsc{lip}-submanifold, one has to   
replace $C^k$-submanifolds and $C^k$-diffeomorphism above  by the 
notions for the respective classes.
\end{definition}

The proof of Theorem \ref{th:lip} for transversally intersecting unions
of closed \textsc{lip}-submanifolds differs only in the construction 
of the paths $\eta_k$. There one has to distinguish different cases,
whether both endpoints lie in different $E_1$, $E_2$ or in the same or even in
both.  The procedure can be extended to finitely many intersecting closed \textsc{lip}-submanifolds. Similar arguments yield that a topologically closed  \textsc{lip}-submanifold with boundary
is permeable. 

\begin{remark}
As unions of affine hyperplanes exhibit transversal intersections, they are permeable. Hence
piecewise Lipschitz functions in the sense of Definition \ref{def:pwlip-md} have permeable exception sets.
\end{remark}

\begin{remark}
Let $f\colon \R\to \R$ be Hölder continuous with Hölder exponent 
$\alpha\in (0,1)$ and let 
\[
\Theta:=\left\{(t,f(t))\colon t\in \R\right\}\,.
\]
An interesting question is: For which $\alpha$, if for any,
is $\Theta$ permeable?

The motivation for this example comes from the standard result from probability theory that almost every path 
of  a  Brownian motion $B$ on some probability space $(\Omega,\cF,(\cF_t)_{t\in [0,1]},P)$ is Hölder continuous with exponent $<\frac{1}{2}$. The graph of such a path constitutes a \textsc{top}-submanifold (with boundary), but not a \textsc{lip}-submanifold\footnote{To see this, assume that there is a bi-\textsc{lip} mapping which locally maps the path to the $x$-axis. Then the path from $x$ to $x+\delta$, $\delta>0$ sufficiently small, would be mapped back to a finite length arc in the Brownian path, which is impossible.}, so Theorem \ref{th:lip} does not apply.
A strong hint that  Brownian paths might not be permeable
is the following: Consider a bounded and progressively measurable process
$H\colon \Omega\times[0,1]\to\R$ 
and an equivalent change of measure from $P$ to $\hat P$ such that $\hat
B_t:=B_t-\int_0^t H_sds$ 
defines a Brownian motion under $\hat P$ (this change of measure exists by Girsanov's theorem). Now the set 
$\{t\in[0,1]\colon \hat B_t=0\}$ is uncountable with probability 
1 under $\hat P$. Therefore 
\[
P\left(\left\{t\in[0,1]\colon B_t=\int_0^t H_s ds\right\} \text{ is countable }
\right)=0\,.
\]
From that we conclude that, for a given $\omega\in \Omega$ the graph of the function $g\colon [0,1]\to \R$ with 
\(
g(t):=\int_0^t H_s(\omega)ds
\)
has uncountable intersection with the graph of $B(\omega)$ almost surely.
A further hint in this direction is Theorem 1.5 in \cite{ruscher}. It states that for every continuous function $g\colon [0,1]\to \R$ the zeros of $B-g$ have Hausdorff dimension at least 
$\frac{1}{2}$ with positive probability.
On the other hand, Theorem 1.2 in \cite{ruscher} states that there 
exists a function $g$, which is Hölder continuous with exponent smaller 
$\frac{1}{2}$ such that $B-g$ has isolated zeros with positive probability
(nothing is said there about the length of the graph of $g$).
See also the related questions in \cite[Open Problem (1)]{peres2010}.
\end{remark}

\section*{Acknowledgements}

Gunther Leobacher is supported by the Austrian
Science Fund (FWF): Project F5508-N26, which is part of the Special Research
Program `Quasi-Monte Carlo Methods: Theory and Applications'.

We are grateful to Tapio Rajala for fruitful discussions and valuable suggestions, and for pointing out Proposition \ref{prop:subset-cont} as well as proposing the example in Remark \ref{rem:question-Rd}.
The authors would like to thank the anonymous referee who provided useful and detailed comments on an 
earlier version of the manuscript.


\subsection*{Author's addresses}

\noindent Gunther Leobacher, Institute of Mathematics and Scientific Computing, University of Graz. 
Heinrichstraße 36, 8010 Graz, Austria. \\{\tt gunther.leobacher@uni-graz.at}\\

\noindent Alexander Steinicke, Institute of Mathematics, 
University of Rostock,
Ulmenstra\ss e 69, Haus 3, 18051 Rostock, Germany. \\{\tt alexander.steinicke@uni-rostock.de}\smallskip

\noindent Alexander Steinicke, Institute of Applied Mathematics, 
Montanuniversitaet Le\-o\-ben.
Peter-Tunner-Straße 25/I, 8700 Leoben, Austria. \\{\tt alexander.steinicke@unileoben.ac.at}
\end{document}